\documentclass[12pt]{article}

\usepackage{amsmath}
\usepackage{amsfonts}
\usepackage{amssymb}
\usepackage{amsthm}
\usepackage{mathrsfs}
\usepackage{graphicx}
\usepackage{enumerate}
\usepackage{xspace}
\usepackage{colortbl}
\usepackage{parskip}
\usepackage{xcolor} 
\usepackage{geometry} 
\usepackage{hyperref}

\theoremstyle{plain}
\newtheorem{lemma}{\textbf{Lemma}}

\newtheorem{remark}{\textbf{Remark}}

\newtheorem{theorem}[lemma]{\textsc{Theorem}}

\newcommand{\be}{\begin{equation}}
\newcommand{\ee}{\end{equation}}

\newcommand{\clf}{{\mathcal F}}

\newcommand{\E}{{\mathbb E}}

\renewcommand{\P}{{\mathbb P}}

\newcommand{\R}{{\mathbb R}}

\newcommand{\Ind}{{\mathbf 1}}

\definecolor{blu}{rgb}{.01,.01,1}
\definecolor{dblu}{rgb}{.01,.01,.85}
\definecolor{dred}{rgb}{.7,.01,.01}
\definecolor{red}{rgb}{1,0,0}
\definecolor{grn}{rgb}{.01,.6,.05}

\newcommand{\bt}{\begin{theorem}}
\newcommand{\et}{\end{theorem}}
\newcommand{\bl}{\begin{lemma}}
\newcommand{\el}{\end{lemma}}

\newcommand{\Bb}{\mathbb B}
\newcommand{\Kb}{\mathbf K}

\newcommand{\Fx}{{\mathbf F}}
\newcommand{\Gx}{{\mathbf G}}
\newcommand{\Hx}{{\mathbf H}}

\newcommand{\CC}{{{\mathbf C}}}

\newcommand{\CW}{{{\mathbf W}}}
\newcommand{\CX}{{{\mathbf X}}}
\newcommand{\CY}{{{\mathbf Y}}}
\newcommand{\CZ}{{{\mathbf Z}}}

\newcommand{\Lb}{{{\bf L}}}

\newcommand{\CV}{\mathbf V}

\newcommand{\xa}{\mathbf a}
\newcommand{\xb}{\mathbf b}
\newcommand{\xu}{\mathbf u}
\newcommand{\xv}{\mathbf v}
\newcommand{\xx}{{\mathbf x}}
\newcommand{\xy}{{\mathbf y}}

\newcommand{\xz}{{\mathbf z}}
\newcommand{\xs}{{\mathbf s}}
\newcommand{\xr}{{\mathbf r}}
\newcommand{\xp}{{\mathbf p}}
\newcommand{\xq}{{\mathbf q}}

\newcommand{\nrm}[1]{\Vert\, #1\,\Vert}

\begin{document} 

\thispagestyle{empty}
\vspace*{15mm}

\begin{center}
    {\large\bfseries Stochastic Approximation in Banach Spaces}\\[2mm]
    {\large\bfseries Without Geometric Constraints}
    
    \vspace{12mm}
    {\large Rajeeva L. Karandikar} \\[1mm]
    and \\[1mm]
    {\large B. V. Rao}
    
    \vspace{6mm}
    {\itshape Chennai Mathematical Institute}\\
    {\itshape H1 Sipcot IT Park, Siruseri, TN 603103, India.}\\ 
    \vspace{2mm}
    {\ttfamily rkarandikar@gmail.com, bvrao@cmi.ac.in}
\end{center}

\vspace{10mm}

\begin{abstract}
The thrust of this article is to show that on all Banach spaces, stochastic approximation holds when the noise sequence is an i.i.d. sequence with mean 0, \emph {without} imposing any condition on the geometry of the space. Also, the same is true when the noise is a sequence of independent random variables under appropriate conditions on the moment. In this case, we need to require that the noise sequence is \emph{tight}.
\end{abstract}

\newpage
\pagestyle{plain}

\section{Introduction}

Let $\Bb$ be a Banach space and 
let $\Gx:\Bb\mapsto \Bb$ be a function such that the equaton $\Gx(x)=0$ has a unique solution (denoted by $\xx^*$). Under appropriate conditions on the function $\Gx$
it can be shown that 
 for any $\xu_0\in\Bb$, the sequence
\[\xu_{n+1}=\xu_n-\Gx(\xu_n),\;\;n\ge 0\]
converges to $\xx^*$ . For example, $\Gx$ could be $G(\xx)=\xx-\Fx(\xx)$ where $\Fx$ is a contraction.

 Consider the scenario where given $\xu_n$, $\Gx(\xu_n)$ is not directly observable, but we can observe %$\eta_n$, which is 
 $\Gx(\xu_n)$ corrupted by noise. Let us write the observed quantity as $\widehat{\Gx}(\xu_n)$ - so that $\Hx_n=\widehat{\Gx}(\xu_n)-{\Gx}(\xu_n)$ is the noise.

One can  modify the successive approximation algorithm used above as follows : Let $\CX_0$ be chosen arbitrarily and, having observed $\CX_n$ at time $n$, instead of taking $\CX_{n+1}=\CX_n-\widehat{\Gx}(\CX_n)$, we take $\CX_{n+1}$ to be a convex combination of $\CX_n-\widehat{\Gx}(\CX_n)$ and $\CX_n$, which leads to the following:  for a sequence of constants $\beta_n\in (0,1)$ converging to zero, we define
 \be\label{U0} \CX_{n+1}=\CX_n-\beta_n\widehat{\Gx}(\CX_n),\;\;n\ge 0,\ee
where now $\Hx_n=\widehat{\Gx}(\CX_n)-{\Gx}(\CX_n)$ is the (observation) noise.
We would like to explore conditions on the function $\Gx$, the noise sequence $\{\Hx_n:n\ge 0\}$, the sequence $\{\beta_n:n\ge 0\}$ (hence forth called the step size) that would ensure that the observed approximation $\{\CX_{n+1}:n\ge 0\}$ converges to $x^*$.

In the finite dimensional case, these and similar questions have been considered for last 75 years, starting with 
Robbins and Monro \cite{RM}, 
Kiefer and Wolfowitz \cite{Kiefer}. Also, see the references given below. 
 
 In \cite{KR} we had considered this in the setting of a Banach space and obtained  conditions, which included conditions on the {\em geometry} of the Banach space, as is customary,  in dealing with limit theorems for sequence of Banach spaces valued random variables.

However, for a given application, the Banach space would be determined by the problem and we cannot change the same. For example, it could be $C[0,1]$, an infinite dimensional linear space that is a Banach space but not a Hilbert space. And  $C[0,1]$ does not even satisfy the Radon-Nikodym property!
 
In this paper, we explore this question when the underlying space is an arbitrary  Banach space $\Bb$. {\em We do not impose any conditions on the geometry of $\Bb$}. The main observation is that the results in \cite{KRV} for the finite dimensional case hold {\em Mutatis Mutandis} when the noise sequence is i.i.d. with zero mean. When we have independent noise, which may not be i.i.d.,  we need to add condition of {\em tightness} of the noise sequence. 

Let us note that when $\Gx(\xx)=(\xx-\mu)$,  $\beta_n=\frac{1}{(n+1)}$, the noise $\{\CW_n:n\ge 1\}$ is i.i.d. with mean 0 and we start with $\CX_0=0$, then it can be seen that $\CX_{n}=\frac{1}{n}(\CW_1+\CW_2+\ldots+\CW_{n})$ and thus convergence of the (almost sure) approximation of the observation sequence is same as the Strong Law of Large Numbers. 
That SLLN holds for arbitrary Banach spaces was proven by J. Hoffmann-Jørgensen \cite{Hoffmann} about 50 years ago!

\section{Main Result}
Here and in what follows, $\nrm{\xx}$ is the norm on the Banach space $\Bb$. We are going to assume that the function $\Gx:\Bb \mapsto \Bb$ satisfies 
\be\label{U2} \text{$\exists \tau>0,\,\rho<1,\; \xx^*\in\Bb$ such that }\;\;\;\nrm{\xx-\xx^*-\textstyle\tau \Gx(\xx)}\le \rho\nrm{\xx-\xx^*}\;\;\;\forall x\in\Bb.\ee
 This of course implies that $\xx^*$ is the {\em unique} solution to $\Gx(x)=0$. As noted in \cite{KR}, the condition \eqref{U2} includes the case (i)  if $\Fx:\Bb\mapsto\Bb$ is a contraction and $\Gx(\xx)=\xx-\Fx(\xx)$, or (ii) if $\exists$ constants $0<c_1<c_2<\infty$ such that  $\forall \xx, \xy\in\Bb$, $\forall \Lb\in\Bb^*$
on has
\begin{equation*}%\label{R3}
(\Lb(\xx-\xy))(\Lb(  \Gx(\xx)-  \Gx(\xy)))\ge 0,
\end{equation*}
\begin{equation*}%\label{R4}
c_1|\Lb(\xx-\xy)|  \le|\Lb(  \Gx(\xx)-  \Gx(\xy))|\le c_2|\Lb(\xx-\xy)|.
\end{equation*} 
Also we assume that the \emph{step size} sequence $\{\beta_n:n\ge 0\}\subseteq(0,1)$ appearing in \eqref{U0} satisfies 
\begin{equation}\label{RM}
\textstyle\lim_{n\rightarrow \infty}\beta_n=0,\;\;\;\sum_{n=0}^\infty\beta_n=\infty.
\end{equation}
Let $\{\CW_n:n\ge 1\}$ be the noise sequence. We assume that its impact on the observation may be inflated by the history of the approximations, namely, observation noise at time $(n+1)$ is  
$\lambda_n(\CX_0,\ldots,\CX_n)\CW_{n+1}$.
We assume that  $\lambda_n:\Bb^{n+1}\mapsto\R$ for $n\ge 0$ satisfies, for some constant $C$ and any $\xu_j\in\Bb$, $0\le j\le n$, $n\ge 0$:
\begin{equation}\label{U5}
|\lambda_n( \xu_0, \xu_1,\ldots, \xu_n)|\le C(1+\max_{0\le k\le n}\nrm{ \xu_k}).\end{equation}
Let $\xx_0\in\R^d$ be fixed, $\CX_0=\xx_0$ and for $n\ge 0$, the stochastic approximation is given by 
\begin{equation}\label{U6}
\CX_{n+1}=\CX_n-\beta_n \bigl(\Gx(\CX_n)+\lambda_n(\CX_0,\ldots,\CX_n)\CW_{n+1}\bigr).\end{equation}
Throughout this paper, we assume that \eqref{U2}, \eqref{RM}, \eqref{U5}, \eqref{U6} are satisfied.\\

Our main result is:
\begin{theorem}\label{Sa-ind}
Suppose that the error $\{\CW_n: n\ge 1\}$ is a sequence of $\Bb$-valued independent random variables.
Suppose that the \emph{step size sequence} $\{\beta_n:n\ge 0\}$ and \emph{the error sequence} $\{\CW_n: n\ge 1\}$ satisfy \emph{any one} of the conditions $(J1)$--$(J3)$ given below \emph{(}in addition to \eqref{U2}, \eqref{RM}, \eqref{U5}, \eqref{U6}\emph{)}. Then 
\be \label{MM} \{\CX_n:n\ge 0\} \text{ converges to }\xx^*\text{ almost surely.}\ee

\begin{itemize}
\item [$(J1)$] Suppose $\{\CW_n: n\ge 1\}$ are i.i.d. with 
\be \label {L1a}\textstyle \E(\nrm{\CW_1}^2)<\infty,\;\;\E(\CW_1)=0,\ee
\be \label {L1b}\textstyle\sum_{n=0}^\infty \beta_n^2<\infty.\ee 
\item [$(J2)$] Suppose $\{\CW_n: n\ge 1\}$ are i.i.d. and for some $\alpha\in [1,2)$ and $0<D<\infty$, we have
\be \label {L2a}\E(\nrm{\CW_1}^\alpha)<\infty,\;\; \E(\CW_1)=0,\ee
\be \label {L2b} \beta_n\le Dn^{-\frac{1}{\alpha}},\;\;\forall n\ge 1.\ee 
\item [$(J3)$] Suppose $\{\CW_n: n\ge 1\}$ is a sequence of independent random variables with mean 0 that is \emph{tight}, \emph{i.e.}, satisfies, $\forall n\ge 1$,
\begin{equation}\label{m1}
\E[\CW_{n}]=0,
\end{equation}
\begin{equation}\label{m2}
\forall \epsilon>0,\,\exists \text{ Compact }\Kb_\epsilon\subseteq\Bb \text{ such that }\P(\CW_{n}\in \Kb_\epsilon)\ge 1-\epsilon.
\end{equation}
Suppose that for some $\alpha\in (1,2]$, we have
\be \label {L3} \sup_{n\ge 1}\E(\nrm{\CW_n}^\alpha)<\infty,\ee
\be \label {L4}\textstyle\sum_{n=0}^\infty \beta_n^\alpha<\infty.\ee 
\end{itemize}
\end{theorem}

\begin{remark} Let us note that \eqref{L2b} and \eqref{L4} imply \eqref{L1b}. Thus in any of the 3 cases, \eqref{L1b} holds. \end{remark}
\begin{remark}
When the Banach space $\Bb$ is $\R^d$ - the finite dimensional Euclidean space, the results follow from \cite{KRV}, Theorem 4 and Theorem 8, since for $\R^d$-valued random variables \eqref{L3} implies \eqref{m2}. Thus, here we extend the result $($for independent mean 0 noise sequence$)$ and we do not impose any conditions on the Banach space $\Bb$ but instead only require that the noise sequence is \emph{tight}. 
\end{remark}
\begin{remark}
It was noted by J. Hoffmann-J\o rgensen \cite{Hoffmann} that while SLLN holds on all Banach Spaces for i.i.d. sequences, once we move away from identical distribution, it may not hold even if we have bounded $L^2$ norms. While the standard approach has been to impose appropriate geometric conditions on the underlying Banach Space, here we show that the probabilistic condition of tightness is enough to get that SLLN holds.
\end{remark} 

Before we proceed to prove this result, we need some results on the convergence of Banach-valued sequences, which may be of independent interest.

\section{Deterministic sequences} 
Throughout this section, let $\{\gamma_n:n\ge 0\}$ be a fixed $[0,\infty)$-valued sequence of constants such that 
\begin{equation}\label{RMa}
\textstyle\lim_{n\rightarrow\infty}\gamma_n=0,\;\;\;\sum_{n=0}^\infty\gamma_n=\infty.
\end{equation} 
Let $m^*<\infty$ be such that $\gamma_n<1$ for all $n\ge m^*$. Such an integer exists since $\gamma_n$ converges to 0, in view of \eqref{RMa}. In the rest of the section, we fix such an $m^*$.
Let us record a simple fact: for all $m,n$ such that $m^*\le m\le n$, we have
\begin{equation}\label{w7}
\textstyle\bigl[\prod_{k=m}^{n}(1-\gamma_k)\bigr]+\sum_{k=m}^{n}\bigl[\prod_{_{\{j:k<j\le n\}}}(1-\gamma_j)\bigr]\gamma_k=1. 
\end{equation}
This can be verified by induction on $n$. In this article, we will follow the usual convention that a product over an empty set is taken as 1. 

Let $\mathscr{L}(\{\gamma_\centerdot\})$ denote the set of all $\Bb$-valued sequences $\{\xz_n:n\ge 1\}$ such that for every $\epsilon>0$, $\exists$ a $\Bb$-valued sequence $\{ \xv_n:n\ge 1\}$, and a $[0,\infty)$-valued sequence $\{\delta_n:n\ge 1\}$ such that
\begin{align}
\label{k1}& \textstyle\sum_{k=0}^n\gamma_k (\xz_{k+1}-\xv_{k+1}), \;\;n\ge 0,\;\;\text{is a $\Bb$-valued convergent sequence}\\
\label{k2} &\textstyle\sum_{k=0}^n\gamma_k(\nrm{ \xv_{k+1}}-\delta_{k+1}), \;\; n\ge 0,\;\;\text{is a $[0,\infty)$-valued convergent sequence}\\
\label{k3}&\;\;0\le \delta_n\le \epsilon.
\end{align}
Clearly, $\mathscr{L}(\{\gamma_\centerdot\})$ is a linear space. 
We start with some auxiliary observations.

\bl\label{deterministic-1}
Let $\{\xa_k:m\le k\le n\}\subseteq \Bb$ be such that 
\be\label{w2}
\nrm{\textstyle\sum_{k=i}^{j}\xa_{k}}\le C\;\;\;\;\;\;\;\;\;\forall m\le i\le j\le n\ee
for some constant $C<\infty$ and let $\{\theta_k:m\le k\le n\}\subset (0,\infty)$ be an increasing sequence.
Then
\be\label{w3}
\nrm{\textstyle\sum_{k=i}^{n}\theta_k\xa_{k}}\le \theta_nC\;\;\;\;\;\;\;\;\;\forall m\le i\le n\ee
\el
\begin{proof} To prove \eqref{w3}, let us fix $i$, $m\le i\le n$. Let $\pi_i=\theta_i$ and for $i<k\le n$, let $\pi_k=\theta_k-\theta_{k-1}$. Then $\pi_k\ge 0$ and $\theta_k=\sum_{j=i}^{k}\pi_j$. Thus
\[\begin{split}\textstyle\sum_{k=i}^{n}\theta_k\xa_{k}&=\textstyle\sum_{k=i}^{n}\sum_{j=i}^{k}\pi_j\xa_{k}\\
&=\textstyle\sum_{j=i}^{n}\sum_{k=j}^{n}\pi_j\xa_{k}\\
&=\textstyle\sum_{j=i}^{n}\pi_j\sum_{k=j}^{n}\xa_{k}
\end{split}\]
and hence using $\pi_j\ge 0$,
\[\begin{split}\textstyle\nrm{\sum_{k=i}^{n}\theta_k\xa_{k}}&\le\textstyle\sum_{j=i}^{n}\nrm{\pi_j\sum_{k=j}^{n}\xa_{k}}\\
&\le\textstyle\sum_{j=i}^{n}\pi_j\nrm{\sum_{k=j}^{n}\xa_{k}}\\
&\le \textstyle\sum_{j=i}^{n}\pi_jC\\
&\le \theta_nC.
\end{split}\]
\end{proof}

\begin{lemma}\label{deterministic-2}
Let $\{\xz_n:n\ge 1\} \subseteq \Bb$, $s_0\in\Bb$, $\{\nu_n:n\ge 1\} \subseteq [1,\infty)$ be an increasing sequence and let $\{\xs_n:n\ge 1\} \subseteq \Bb$ be defined by, for $n\ge 0$ 
\be\label{B1}
\xs_{n+1}=(1- \gamma_n)\xs_n+\gamma_n\nu_n\xz_{n+1}.\ee
Then we have
\be\label{B1x}\textstyle
\limsup_n\frac{1}{\nu_n}\nrm{\xs_{n+1}}\le \limsup_n\nrm{\xz_{n+1}}.\ee
Further, if $n\mapsto\sum_{k=0}^n\gamma_k\xz_{k+1}$ is a convergent sequence $($in $\nrm{\cdot})$, then 
\be\textstyle
\label{B2} \lim_n\frac{1}{\nu_n}\nrm{\xs_{n+1}}=0.\ee
\end{lemma}
\begin{proof} 
Let us note that for $m\le n$
 \be\textstyle
\label{B3} \xs_{n+1}=\bigl[\prod_{k=m}^n(1- \gamma_k)\bigr]\xs_m+\sum_{k=m}^n\bigl[\prod_{j=k+1}^n(1- \gamma_j)\bigr]\gamma_k \nu_k\xz_{k+1}.\ee
For any $\eta> \limsup_n\nrm{\xz_{n+1}}$, and for $m\ge m^*$ such that for $j,k\ge m$, one has $\gamma_j\le 1$ and $\nrm{\xz_{k+1}}\le \eta$, we have 
\[
\begin{split}
\nrm{\xs_{n+1}}\le&\textstyle\bigl[\prod_{k=m}^n(1- \gamma_k)\bigr]\nrm{\xs_m}+ \sum_{k=m}^n\bigl[\prod_{j=k+1}^n(1- \gamma_j)\bigr] \gamma_k \nu_k\nrm{\xz_{k+1}}\\
&\le \textstyle\bigl[\prod_{k=m}^n(1- \gamma_k)\bigr]\nrm{\xs_m}+ \nu_n\eta\sum_{k=m}^n\bigl[\prod_{j=k+1}^n(1- \gamma_j)\bigr] \gamma_k .
\end{split}\]
It follows from \eqref{w7} that $\sum_{k=m}^n\bigl[\prod_{j=k+1}^n(1- \gamma_j) \gamma_j\bigr]\le 1$ and hence 
\[
\textstyle\frac{1}{\nu_n}\nrm{\xs_{n+1}} \textstyle\le \frac{1}{\nu_m}\nrm{\xs_m} \bigl[\prod_{k=m}^n(1- \gamma_k)\bigr]+\eta .
\]
In view of \eqref{RMa}, for $m\ge m^*$ fixed, $\bigl[\prod_{k=m}^n(1- \gamma_k)\bigr]$ converges to 0 as $n$ increases to $\infty$ and hence we have
\be\label{B5} \textstyle
\limsup_n\frac{1}{\nu_n}\nrm{\xs_{n+1}}\le \eta.\ee
This proves \eqref{B1x}. 

For the second part, using $n\mapsto\sum_{j=0}^n\gamma_j\xz_{j+1}$ is a convergent sequence, given $\eta>0$, we choose $m\ge m^*$ such that for $j,k\ge m$ we have $\nrm{\sum_{i=j}^k\gamma_i\xz_{i+1}}\le \eta$. Fix $n\ge m$. 
Note that $k\mapsto \bigl[\prod_{j=k+1}^n(1- \gamma_j)\bigr]$ is increasing and bounded by 1, thus it follows that $k\mapsto \bigl[\prod_{j=k+1}^n(1- \gamma_j)\bigr]\nu_k$ is increasing and bounded by $\nu_n$. Thus invoking Lemma \ref{deterministic-1} it follows that
\be\label{B7} \textstyle
\nrm{\sum_{k=m}^n\bigl[\prod_{j=k+1}^n(1- \gamma_j)\bigr]\gamma_k \nu_k\xz_{k+1}}\le \nu_n\eta.
\ee
Thus, for $n\ge m$
\[\begin{split}
\textstyle
\frac{1}{\nu_n}\nrm{\xs_{n+1}}&\le\textstyle\frac{1}{\nu_n}\bigl[\prod_{k=m}^n(1- \gamma_k)\bigr]\nrm{\xs_m}+\frac{1}{\nu_n}\left( \sum_{k=m}^n\bigl[\prod_{j=k+1}^n(1- \gamma_j)\bigr] \gamma_k \nu_k\nrm{\xz_{k+1}}\right)\\
&\le \textstyle\frac{1}{\nu_n}\bigl[\prod_{k=m}^n(1- \gamma_k)\bigr]\nrm{\xs_m}+\eta\end{split}\]
As seen above, $\frac{1}{\nu_n}\bigl[\prod_{k=m}^n(1- \gamma_k)\bigr]\nrm{\xs_m}$ converges to zero and hence we conclude
\[
\textstyle\limsup_{n\rightarrow\infty} \frac{1}{\nu_n}\nrm{\xs_{n+1}}\le \eta.\]
This holds for all $\eta>0$ and as a consequence, we conclude that \eqref{B2} holds.
\end{proof}

\bl\label{deterministic-T}
Let $\{\xz_n:n\ge 1\} \in\mathscr{L}(\{\gamma_\centerdot\})$, $\{\nu_n:n\ge 1\} \subseteq [1,\infty)$ be an increasing sequence, $\xs_0\in\Bb$ and let $\{\xs_n:n\ge 1\} \subseteq \Bb$ be defined by, for $n\ge 0$ 
\be\label{B1_v2}
\xs_{n+1}=(1- \gamma_n)\xs_n+\gamma_n\nu_n\xz_{n+1}.\ee
Then \be\textstyle
\label{B20} \lim_n\frac{1}{\nu_n}\nrm{\xs_{n+1}}=0.\ee
\el
\begin{proof} Let us fix $\epsilon>0$ and let $\{ \xv_n:n\ge 1\}$, $\{\delta_n:n\ge 1\}$ be such that \eqref{k1}--\eqref{k3} are satisfied. Let $\{ \xu_n:n\ge 1\}, \{ \xp_n:n\ge 1\}, \{ \xq_n:n\ge 1\}$ be defined by $ \xu_{n}=\xz_{n}- \xv_{n}$, $ \xp_0=\xs_0$ and $ \xq_0=0$ and for $n\ge 0$ let
\be\label{B30} \xp_{n+1}=(1- \gamma_n) \xp_n+\gamma_n\nu_n \xu_{n+1},\ee
\be\label{B31} \xq_{n+1}=(1- \gamma_n) \xq_n+\gamma_n\nu_n \xv_{n+1}.\ee
Let us note that $\xs_{n+1}= \xp_{n+1}+\xq_{n+1}$ and thus
\be\label{B32} \nrm{\xs_{n+1}}\le \nrm{ \xp_{n+1}}+\nrm{ \xq_{n+1}}.\ee
In view of assumption \eqref{k1} and Lemma \ref{deterministic-2}, it follows that 
\be\label{B33}\textstyle \lim_n\frac{1}{\nu_n}\nrm{ \xp_{n+1}}=0.\ee

For $n\ge m$, let $\xi_{n}= \nrm{ \xv_n}-\delta_{n}$ (and with $\zeta_m=0$, $\omega_m=0$)
\be\label{B35}\zeta_{n+1}=(1- \gamma_n) \zeta_n+\gamma_n\nu_n \xi_{n+1},\ee
\be\label{B36}\omega_{n+1}=(1- \gamma_n) \omega_n+\gamma_n\nu_n \delta_{n+1}.\ee
Let us note that for $n\ge m\ge m^*$
\[\begin{split}
\nrm{ \xq_{n+1}}&\le\textstyle\bigl[\prod_{k=m}^n(1- \gamma_k)\bigr]\nrm{ \xq_m}+\sum_{k=m}^n\bigl[\prod_{j=k+1}^n(1- \gamma_j)\bigr]\gamma_k \nu_k\nrm{ \xv_{k+1}}\\
&\le\textstyle\bigl[\prod_{k=m}^n(1- \gamma_k)\bigr]\nrm{ \xq_m}+\sum_{k=m}^n\bigl[\prod_{j=k+1}^n(1- \gamma_j)\bigr]\gamma_k \nu_k(\zeta_{k+1}+\delta_{k+1})\\
&\le \textstyle\bigl[\prod_{k=m}^n(1- \gamma_k)\bigr]\nrm{ \xq_m}+\zeta_{n+1}+\omega_{n+1}
\end{split}\]
Once again using Lemma \ref{deterministic-2} along with \eqref{k2} and \eqref{k3} , it follows that 
\[\textstyle \lim_n\frac{1}{\nu_n}\zeta_{n+1}=0,\]
\[\textstyle\limsup_n \frac{1}{\nu_n}\omega_{n+1}\le\epsilon\]
and using $\lim_n\frac{1}{\nu_n}\bigl[\prod_{k=m}^n(1- \gamma_k)\bigr]\nrm{ \xq_m}=0$, we conclude
\be\label{B34} \textstyle\limsup_n\frac{1}{\nu_n}\nrm{ \xq_{n+1}}\le\epsilon\ee
In view of \eqref{B32}, \eqref{B33}, \eqref{B34}, we conclude that 
\[\limsup_n \frac{1}{\nu_n}\nrm{\xs_{n+1}}\le\epsilon.\]
Since this holds for all $\epsilon>0$, the result \eqref{B20} follows.
\end{proof}

Here is our main observation on $\Bb$-valued (deterministic) sequences, that relates to stochastic approximation.
\bt \label{main-discrete}
Let $\{\beta_n:n\ge 0\}$ satisfy \eqref{RM} and $\Gx:\Bb\mapsto \Bb$ satisfy \eqref{U2}. 
Let $\{\xz_n:n\ge 1\} \in\mathscr{L}(\{\gamma_\centerdot\})$. Let $\xx_0\in\Bb$ be arbitrary and let $\{\xx_{n+1}:n\ge 0\}\subseteq \Bb$ be defined by 
\be\label{w6} \xx_{n+1}=\xx_n-\beta_n\Gx(\xx_n)+\beta_n\phi_n(\xx_0,\xx_1,\ldots ,\xx_n) \xz_{n+1},\ee
where 
\be\label{w9}  \phi_n(\xx_0,\xx_1,\ldots ,\xx_n)=\bigl(1+\max_{\{0\le k\le n\}}\nrm{\xx_k}\bigr).\ee
Then 
\[ \lim_{n\rightarrow \infty} \nrm{\xx_n-\xx^*}=0.\]
\et
\begin{proof} Let $\Gx:\Bb\mapsto \Bb$, $\tau >0,\;\rho \in (0,1)$, be such that \eqref{U2} holds. Let $\nu_n=\tau\phi_n(\xx_0,\xx_1,\ldots ,\xx_n)$, $\gamma_n=\tau^{-1}\beta_n$, and $\xr_n=\xx_n-\xx^*-\tau \Gx(\xx_n)$. With these notations, the condition \eqref{U2} yields
\be\label{U7}\nrm{\xr_n}\le\rho\nrm{\xx_n-\xx^*}.\ee
The equation \eqref{w6} can be rewritten as 
\be\label{w12}
\begin{split} 
\xx_{n+1}-\xx^*&=\xx_n-\xx^*-\gamma_n\tau \Gx(\xx_n)+\gamma_n\tau\phi_n(\xx_0,\xx_1,\ldots ,\xx_n) \xz_{n+1}\\
&=(1-\gamma_n)(\xx_n-\xx^*)+\gamma_n\xr_n+\gamma_n\nu_n \xz_{n+1}.\end{split}
\ee
Note that, we have fixed $m^*$ such that $\gamma_n=\tau^{-1}\beta_n<1$ for all $n\ge m^*$.
For $m^*\le m\le n$, let
 \begin{align}\label{w13a}
  \xa_{m,n+1}&= \textstyle\bigl[\prod_{j=m}^{n}(1-\gamma_j)\bigr](\xx_m-\xx^*)\\
  \label{w13b}
\xb_{m,n+1}&= \textstyle\sum_{k=m}^{n} \bigl[\prod_{_{\{j:k<j\le n\}}}(1-\gamma_j)\bigr]\gamma_k\xr_k\\
\label{w13c}
 \xs_{m,n+1}&= \textstyle{\sum_{k=m}^n}\bigl[\prod_{_{\{j:k<j\le n\}}}(1-\gamma_j)\bigr]\gamma_k\nu_k \xz_{k+1}
  \end{align}
By induction on $n$, one can check that
\be\label{w14} \xx_{n+1}-\xx^*=\xa_{m,n+1}+\xb_{m,n+1}+\xs_{m,n+1}\ee
Let us choose $0<\epsilon<\frac{1}{2}$ such that 
\be\label{w15} (1+2\epsilon)\rho<1.\ee
This can be done as $\rho<1$. Noting that
\[\xs_{0,n+1}=(1- \gamma_n)\xs_{0,n}+\gamma_n\nu_n\xz_{n+1}\]
Invoking Theorem \ref{deterministic-T}, we have
\be\label{w16} \lim_n\textstyle\frac{1}{\nu_n}\nrm{\xs_{0,n}}=0.\ee
Using this, let us get $m\ge m^*$ such that $n\ge m$ implies
\[\textstyle \frac{1}{\nu_n}\nrm{\xs_{0,n}}\le \textstyle\frac{1}{2\tau}\epsilon.\]
It can be seen (by induction) that for $n \ge m$
\[ \xs_{0,n+1}=\textstyle\bigl[\prod_{j=m}^{n}(1-\gamma_j)\bigr]\xs_{0,m}+\xs_{m,n+1}\]
and hence for $n\ge m$
\[ \begin{split}
 \nrm{\xs_{m,n+1}}&\le \nrm{\xs_{0,n+1}}+\textstyle\bigl[\prod_{j=m}^{n}(1-\gamma_j)\bigr]\nrm{\xs_{0,m}}\\
&\le \nrm{\xs_{0,n+1}}+\nrm{\xs_{0,m}}.\end{split}\]
and thus
 \begin{equation}\label{w21}
\textstyle \frac{1}{\nu_n} \nrm{\xs_{m,n+1}}\le \textstyle\frac{1}{\tau}\epsilon,\;\;\forall n\ge m .\ee
 
For $n\ge 0$, let $ \pi_n=\max\{\nrm{\xx_i-\xx^*}: 0\le i\le n\}$. Let us note that 
 \begin{equation}\label{w22}
 \nu_k=\tau(1+\max\{\nrm{\xx_i}:0\le i\le k\})\le \tau(1+\nrm{\xx^*}+\pi_k)\ee
We will now prove that $\forall k\ge 1$
 \begin{equation}\label{w23}
 \pi_k\le (1+2\epsilon)(1+\nrm{\xx^*}+\pi_m).\ee
Of course \eqref{w23} holds for all $k\le m$ since $k\mapsto\pi_k$ is increasing. We will prove by induction that if \eqref{w23} holds for $k=n$, it is also true for $k=n+1$.
Using \eqref{w13a}, \eqref{w13b} and \eqref{w14} we get
 \begin{equation}\label{w27}
\nrm{\xx_{n+1}-\xx^*}\le \nrm{\xa_{m,n+1}}+\nrm{\xb_{m,n+1}}+\nrm{\xs_{m,n+1}}\ee
Clearly
 \begin{equation}\label{w28}
\nrm{\xa_{m,n+1}}\le \textstyle\bigl[\prod_{j=m}^{n}(1-\gamma_j)\bigr]\pi_m.\ee
Using \eqref{U7} we get
 \begin{equation}\label{w28x}\begin{split}
\nrm{\xb_{m,n+1}}&\le\textstyle\sum_{k=m}^{n} \bigl[\prod_{_{\{j:k<j\le n\}}}(1-\gamma_j)\bigr]\gamma_k\nrm{\xr_k}\\
&\le \textstyle\sum_{k=m}^{n} \bigl[\prod_{_{\{j:k<j\le n\}}}(1-\gamma_j)\bigr]\gamma_k\rho\pi_k
\end{split}\ee
Using the induction hypothesis (that \eqref{w23} is valid for $k\le n$), \eqref{w28x} and \eqref{w15} we get
 \begin{equation}\label{w29}\begin{split}
\nrm{\xb_{m,n+1}}&\le \textstyle\sum_{k=m}^{n} \bigl[\prod_{_{\{j:k<j\le n\}}}(1-\gamma_j)\bigr]\gamma_k\rho(1+2\epsilon)(1+\nrm{\xx^*}+\pi_m)\\
&\le \textstyle\sum_{k=m}^{n} \bigl[\prod_{_{\{j:k<j\le n\}}}(1-\gamma_j)\bigr]\gamma_k(1+\nrm{\xx^*}+\pi_m)\end{split}\ee
Invoking \eqref{w28}, \eqref{w29} along with \eqref{w7} we deduce
\be\label{w32}
\nrm{\xa_{m,n+1}}+\nrm{\xb_{m,n+1}}\le(1+\nrm{\xx^*}+\pi_m)\ee
In view of \eqref{w21}, \eqref{w22} and the induction hypothesis \eqref{w23} we get
\be\label{w33_v2} \begin{split}\nrm{\xs_{m,n+1}}&\le \textstyle\frac{1}{\tau}\epsilon\,\nu_n\\
&\le\textstyle\frac{1}{\tau}\epsilon\,\tau(1+\nrm{\xx^*}+\pi_n)\\
&\le \textstyle\epsilon\,(1+\nrm{\xx^*}+(1+2\epsilon\,)\pi_m)\\
&\le \textstyle2\epsilon\,(1+\nrm{\xx^*}+\pi_m)\\
\end{split}\ee
since $2\epsilon^2\le \epsilon$ (in view of the fact that $0<\epsilon<\frac{1}{2}$). The inequalities \eqref{w32} and \eqref{w14} show that if \eqref{w23} is true for $k\le n$, then it is true for $k=n+1$. Hence \eqref{w23} is true for all $k\ge 1$.
Thus we conclude that $\sup_n\pi_n<\infty$ and as a consequence that $\sup_{n\ge 1} \nu_n<\infty.$
Thus we can now deduce from Theorem \ref{deterministic-T} that for each $m\ge 1$,
 \be\label{w36}
\lim_{n\rightarrow \infty}\textstyle\nrm{\xs_{m,n}}=0 .\ee
Using the definition \eqref{w13a} and the assumption that $\sum_n\gamma_n=\infty$, it follows that $\forall m\ge 1$,
 \be\label{w38}
\textstyle\lim_n\nrm{\xa_{m,n}}=0.\ee
Let $ \chi_n=\sup\{\nrm{\xx_i-\xx^*}: i\ge n\}$ and $\chi^*=\lim_n\chi_n$. Note that $n\mapsto \chi_n$ is decreasing and $\chi_{_1}=\lim_m\pi_m<\infty$ as seen above, hence $\lim_n\chi_n$ exists and is finite. 

Using \eqref{w28x}, we see that for each $m$, 
 \begin{equation}\label{w41}\begin{split}
\nrm{\xb_{m,n+1}}&\le\textstyle\sum_{k=m}^{n} \bigl[\prod_{_{\{j:k<j\le n\}}}(1-\gamma_j)\bigr]\gamma_k\rho\pi_k\\
&\le\rho\chi_m\textstyle\sum_{k=m}^{n} \bigl[\prod_{_{\{j:k<j\le n\}}}(1-\gamma_j)\bigr]\gamma_k\\
&\le \rho\chi_m\end{split}\ee
Using \eqref{w27} along with \eqref{w36}, \eqref{w38} and \eqref{w41}, we conclude that for each $m\ge m^*$,
\[\begin{split}
\textstyle \limsup_n\nrm{\xx_{n+1}-\xx^*}&\le \textstyle\limsup_n\nrm{\xa_{m,n+1}}+\limsup_n\nrm{\xb_{m,n+1}}+\limsup_n\nrm{\xs_{m,n+1}}\\
 &\le 0+\rho\chi_m+0.\end{split}\]
Thus we have
 \[\chi^*\le\rho\chi_m\;\;\;\forall m.\]
Since $\chi_m\downarrow\chi^*$ we conclude $\chi^*\le\rho\chi^*$. Since $\rho<1$ and $0\le\chi^*<\infty$ as noted above, this is possible only if $\chi^*=0$.
\end{proof}

\section{Proof of the Main Theorem}

We begin with a result on tight and uniformly integrable Banach-valued sequences of random variables, which plays a key role in our approach. This could be of independent interest.
\begin{theorem}\label{Aux1}
Let $\{\CW_n:n\ge 1\}$ be a sequence of independent $\Bb$-valued random variables with mean zero and is \emph{tight} $($\emph{i.e.} satisfies \eqref{m1}, \eqref{m2}$)$.
Suppose $\{\CW_n:n\ge 1\}$ is uniformly integrable, \emph{i.e.}
\begin{equation}\label{m3}
\textstyle\lim_{t\rightarrow\infty}\Bigl[ \sup_{n\ge 1}\E\bigl[\nrm{\CW_{n}}\Ind_{\{\nrm{\CW_{n}}\ge t\}}\bigr]\Bigr]=0.
 \end{equation}
Then for all $\epsilon>0$ there exists $t\ge 1$, $\{\xa_j: 1\le j\le t\}\subseteq\Bb$, a sequence $\{ \CY_{n}:n\ge 1\}$ of $\Bb$-valued independent random variables, a constant $C<\infty$ and for $1\le j\le t$, a sequence $\{\xi_{n,j}: n\ge 1\}$ of $[-1,1]$-valued independent random variables such that for all $n\ge 1$
\begin{align}\label{R1}
\CW_n=\CY_n+&\textstyle\sum_{j=1}^t \xa_j\xi_{n,j},\\
\label{R2}
\E[\xi_{n,j}]&=0,\;\;\;\;\forall\,1\le j\le t,\\
\label{R3}\E[\nrm{\CY_n}]&\le \epsilon\\
\label{R4}
\nrm{\CY_n-\CW_n}&\le C.
\end{align}
Further, if $\{\CW_n:n\ge 1\}$ are i.i.d. then $\{\CY_n:n\ge 1\}$ are i.i.d. and for each $1\le j\le t$, $\{\xi_{n,j}: n\ge 1\}$ are also i.i.d.
\end{theorem}
\begin{proof} Given $\epsilon>0$, using uniform integrability of $\{\CW_n:n\ge 1\}$, let us get $\lambda>0$ such that 
\be\label{w7q}
\E[ \nrm{\CW_n}\Ind_{\{\nrm{\CW_n}>\lambda\}}]\le \textstyle\frac{1}{8}\epsilon\ee
and then using tightness of $\{\CW_n:n\ge 1\}$ get compact $\Kb\subseteq\Bb$ such that
\be\label{w8q}
\P(\CW_{n+1}\not\in \Kb)\le \textstyle\frac{1}{8\lambda}\epsilon\;\;\forall n\ge 0.\ee
Using \eqref{w7q} and \eqref{w8q} in the last step, it follows that
\be\label{w9q}\begin{split}
\E[ \nrm{\CW_{n+1}} \Ind_{\{\CW_{n+1}\not\in \Kb\}}]\le &\E[ \nrm{\CW_{n+1}} \Ind_{\{\CW_{n+1}\not\in \Kb\}}\Ind_{\{\nrm{\CW_{n+1}}\le\lambda\}}]+ \E[ \nrm{\CW_{n+1}} \Ind_{\{\CW_{n+1}\not\in \Kb\}}\Ind_{\{\nrm{\CW_{n+1}}>\lambda\}}]\\
\le&\lambda \E[ \Ind_{\{\CW_{n+1}\not\in \Kb\}}]+\E[ \nrm{\CW_{n+1}} \Ind_{\{\nrm{\CW_{n+1}}>\lambda\}}]\\
\le &\lambda\textstyle\frac{1}{8\lambda}\epsilon+\frac{1}{8}\epsilon=\frac{1}{4}\epsilon.\end{split}\ee
Let us note that in view of $\E[\CW_{n+1}]=0$, it follows that
\be\label{w10q}\begin{split}
\nrm{\E[\CW_{n+1}\Ind_{\{\CW_{n+1}\}\in \Kb\}} ]}&=\nrm{\E[-\CW_{n+1} \Ind_{\{\CW_{n+1}\not \in \Kb\}} ]}\\
&\le \E[\nrm{\CW_{n+1}}\Ind_{\{\CW_{n+1}\not \in \Kb\}}]\\
&\le\textstyle \frac{1}{4}\epsilon.\\
\end{split}\ee
Using compactness of $\Kb$, we can get $t$, $\xa_j\in \Kb$, $\CC_j\subset \Bb$ for $1\le j\le t$ such that $\{\CC_1,\CC_2,\ldots \CC_t\}$ is a (Borel measurable) partition of $\Kb$ and 
\be\label{w11q}
\nrm{\xy-\xa_j}\le\textstyle \frac{1}{4}\epsilon,\;\;\forall \xy\in \CC_j,\; 1\le j\le t.\ee Let $\eta_{n+1,j}=\Ind_{\{\CW_{n+1}\in \CC_j\}}$ and $\mu_{n+1,j}=\E[\eta_{n+1,j}]=\P\{\CW_{n+1}\in \CC_j\}$. 
In view of \eqref{w11q} we get
\be\label{w11qq}
\nrm{\CW_{n+1} \Ind_{\{\nrm{\CW_{n+1}}\in \Kb\}}-\textstyle\sum_{j=1}^t \xa_j\eta_{n+1,j}}\le\textstyle \frac{1}{4}\epsilon.\ee
Combining this with \eqref{w10q}, we conclude
\be\label{w12q}
\nrm{\textstyle\sum_{j=1}^t \xa_j\mu_{n+1,j}}\le \frac{1}{2}\epsilon.\ee
For $n\ge 0$ and $1\le j\le t$, let
 \[\xi_{n+1,j}=\eta_{n+1,j}-\mu_{n+1,j}\] 
 and 
 \[\CY_{n+1}= \CW_{n+1}-\textstyle\sum_{j=1}^t \xa_j\xi_{n+1,j}.\]
By definition, \eqref{R1} and \eqref{R2} are true. Let us note that
\be\label{w12qq}
\CY_{n+1}=\CW_{n+1} \Ind_{\{\nrm{\CW_{n+1}}\not\in \Kb\}}+(\CW_{n+1} \Ind_{\{\nrm{\CW_{n+1}}\in \Kb\}} -\textstyle\sum_{j=1}^t \xa_j\eta_{n+1,j})+\textstyle\sum_{j=1}^t \xa_j\mu_{n+1,j}\ee
Combining \eqref{w9q}, \eqref{w11qq}, \eqref{w12q} and \eqref{w12qq} we get 
\be\label{w13q}\begin{split}
\E[\nrm{\CY_{n+1}}]\le&\;\;\; \E[\nrm{\CW_{n+1} \Ind_{\{\nrm{\CW_{n+1}}\not\in \Kb\}}}] + \E[\nrm{\CW_{n+1} \Ind_{\{\nrm{\CW_{n+1}}\in \Kb\}}\\
&\;\;\;\;\;\;-\textstyle\sum_{j=1}^t \xa_j\eta_{n+1,j}}]+\nrm{\textstyle\sum_{j=1}^t \xa_j\mu_{n+1,j}}\\ \le
&\;\;\frac{\epsilon}{4}+\frac{\epsilon}{4}+\frac{\epsilon}{2}=\epsilon.
  \end{split}\ee
This proves \eqref{R3}. And \eqref{R4} follows with $C=\sum_{j=1}^t\nrm{\xa_j}.$

From the construction, it is clear that since $\{\CW_n:n\ge 1\}$ is a sequence of independent random variables, so are $\{\CY_n:n\ge 1\}$ as well as $\{\xi_{n,j}: n\ge 1\}$ for each $j$, $1\le j\le t$. And when $\{\CW_n:n\ge 1\}$ are i.i.d., $\{\CY_n:n\ge 1\}$ as well as $\{\xi_{n,j}: n\ge 1\}$ for each $j$, $1\le j\le t$ are also i.i.d. 
\end{proof}
\begin{remark}\label{integrability}
As a consequence of \eqref{R4} we note that for any $\alpha>0$
\be\label{U24}
\sup_{n\ge 1}\E[\nrm{\CW_n}^\alpha]<\infty\;\text{ if and only if }\; \sup_{n\ge 1}\E[\nrm{\CY_n}^\alpha]<\infty.\ee
\end{remark}

Coming to the proof of Theorem \ref{Sa-ind}, we need some more notation.
Let $\Gx$, $\{\beta_n:n\ge 0\}$, $\{\CW_n: n\ge 1\}$, and $\{\lambda_n:n\ge 0\}$ be as in the statement of Theorem \ref{Sa-ind} and $\{\phi_n:n\ge 0\}$ be given by \eqref{w9}. For $n\ge 0$, let
\be\label{w33_v3}
 \psi_n(\xx_0,\xx_1,\ldots ,\xx_n)=\lambda_n(\xx_0,\xx_1,\ldots ,\xx_n)\bigl(C_1\phi_n(\xx_0,\xx_1,\ldots ,\xx_n)\bigr)^{-1}\ee
where $C_1$ is the constant in condition \eqref{U5} on $\lambda_n$. 
Let $\clf_n=\sigma(\CX_0,\CX_1,\ldots ,\CX_n)$, $\CZ_0=\CX_0=\xx_0$ and $\{\CZ_{n+1}: n\ge 0\}$ be defined by, for $n\ge 0$,
\be\label{w33_v4}
\mathrm{\CZ}_{n+1}=\psi_n(\CX_0,\CX_1,\ldots ,\CX_n)\CW_{n+1}.
\ee
Note that $\{\psi_n(\CX_0,\CX_1,\ldots ,\CX_n):n\ge 0\}$ are $(\clf_n)$-adapted random variables, bounded by 1.
We can see that \eqref{U6} can be recast as 
\begin{equation}\label{V6}
\CX_{n+1}=\CX_n-\beta_n \bigl(\Gx(\CX_n)+\phi_n(\CX_0,\ldots,\CX_n)\CZ_{n+1}\bigr).\end{equation}
We will show that under each of the conditions $(J1)$, $(J2)$, or $(J3)$,
\begin{equation}\label{V7}\P(\{\CZ_n: n\ge 1\}\subseteq \mathscr{L}(\{\beta_\centerdot\}))=1.\ee
Invoking Theorem \ref{main-discrete} the required result \eqref{MM} would follow from this. 

\textbf{Step 1:}\\
This step is common when $(J1)$, $(J2)$ or $(J3)$ holds. It is easy to see that under each of these assumptions, it follows that $\{\CW_n: n\ge 1\}$ is tight and uniformly integrable. 

Let us fix $\epsilon>0$. So we are given that $\{\CW_n:n\ge 1\}$ is a \emph{tight} sequence of independent random variables with mean 0. Using Theorem \ref{Aux1}, we get $t\ge 1$, $\{\xa_j: 1\le j\le t\}\subseteq\Bb$, a sequence $\{ \CY_{n}:n\ge 1\}$ of $\Bb$-valued independent random variables, a constant $C<\infty$ and for $1\le j\le t$, a sequence $\{\xi_{n,j}: n\ge 1\}$ of $[-1,1]$-valued independent random variables such that \eqref{R1}, \eqref{R2}, \eqref{R3}, \eqref{R4} hold. Note that
\be\label{R7}
\begin{split}
\CZ_{n+1}&=\psi_n(\CX_0,\CX_1,\ldots ,\CX_n)\CW_{n+1}\\
&=\psi_n(\CX_0,\CX_1,\ldots ,\CX_n)\bigl(\CY_{n+1}+\textstyle\sum_{j=1}^t \xa_j\xi_{n+1,j}\bigr).\end{split}\ee
For $n\ge 0$, let
\be\label{R8}
\CV_{n+1}=\psi_n(\CX_0,\CX_1,\ldots ,\CX_n)\CY_{n+1}.\ee
Recall that $\{\psi_n(\CX_0,\CX_1,\dots ,\CX_n):n\ge 0\}$ are $(\clf_n)$-adapted random variables, bounded by 1. Hence for each $j$,
\[\textstyle\sum_{n=0}^m\beta_n\psi_n(\CX_0,\CX_1,\ldots ,\CX_n)\xi_{n+1,j}\]
is a real-valued $L^2$-bounded martingale (since $\{\beta_n\}$ satisfies \eqref{RM}) and hence converges almost surely. Thus the $\Bb$-valued sequence of random variables 
\be\label{R10}
\textstyle\sum_{k=0}^n\beta_k(\CZ_{k+1}-\CV_{k+1}) =\sum_{j=1}^t\xa_j\left(\sum_{k=0}^n\beta_k\psi_k(\CX_0,\CX_1,\ldots ,\CX_k)\xi_{k+1,j}\right)\;\;
\text{converges almost surely.}\ee

\textbf{Step 2:}\\
Let us assume that $(J1)$ is true. For $n\ge 1$, let $\delta_n=\E(\nrm{\CV_{n+1}}).$ Then
\be\label{S14}\delta_n\le\E(\nrm{\CV_{n+1}}) \le \E(\nrm{\CY_{n+1}})\le \epsilon.\ee
Let us note that (using the definition of $\delta_n$ in the first step, and \eqref{R8} in the second step) we have
\be\label{R17_v1}
\begin{split}
\textstyle\sum_{n=0}^\infty\beta_n^2\E[(\nrm{\CV_{n+1}}-\delta_n)^2]&\le \textstyle\sum_{n=0}^\infty\E[\beta_n^2\nrm{\CV_{n+1}}^2]\\
&\le \textstyle\sum_{n=0}^\infty\E[\beta_n^2\nrm{\CY_{n+1}}^2]\\
&\le \textstyle [\sup_{n\ge 0}\E(\nrm{\CY_{n+1}}^2)]\left[\sum_{n=0}^\infty\beta_n^2\right]\\
&<\infty.
\end{split}\ee
Since $\zeta_{n+1}=(\nrm{\CV_{n+1}}-\delta_n)$ is a sequence of independent random variables with mean 0, \eqref{R17_v1} implies that $\sum_{k=0}^n\beta_k\zeta_{k+1}$ is an $L^2$-bounded martingale and hence
\be\label{S18}
\textstyle\sum_{n=0}^\infty\beta_n(\nrm{\CV_{n+1}}-\delta_n) \;\;\text{converges almost surely.}\ee
Invoking \eqref{R10}, \eqref{S14} and \eqref{S18}, it follows that \eqref{V7} holds, completing the proof as noted earlier, when $(J1)$ holds. \\

\textbf{Step 3:}\\
Let us now assume that $(J2)$ is true and $\alpha\in[1,2)$ be such that \eqref{L2a} and \eqref{L2b} are satisfied. Note that in view of Theorem \ref{Aux1} and Remark \ref{integrability}, we have that $\{\CY_j:j\ge 1\}$ are i.i.d. and  
\be\label{S29}
\E(\nrm{\CY_1}^\alpha)<\infty.\ee
For $n\ge 1$, let $\delta_n=\E(\nrm{\CV_{n+1}}\Ind_{\{\nrm{\CY_{n+1}}^\alpha\le n\}})$ and 
\[\zeta_{n+1}=\nrm{\CV_{n+1}}\Ind_{\{\nrm{\CY_{n+1}}^\alpha\le n\}}-\delta_n.\]
From the definitions, it is clear that $\sum_{n=1}^m \beta_n\zeta_{n+1}$ is a martingale. We will first prove that it is $L^2$-bounded.

By definition $\E[\zeta_{n+1}]=0$ and thus, using \eqref{R8} along with the observation that $\psi_n$ is bounded by 1, we get for $n\ge 1$ 
\[\begin{split}
\E[\beta_n^2\zeta_{n+1}^2]&\le \E[\beta_n^2\nrm{\CV_{n+1}}^2\Ind_{\{\nrm{\CY_{n+1}}^\alpha\le n\}}]\\
&\le \E[\beta_n^2\nrm{\CY_{n+1}}^2\Ind_{\{\nrm{\CY_{n+1}}^\alpha\le n\}}]\\
&\le \E[\beta_n^2n^{\frac{2}{\alpha}}\Ind_{\{\nrm{\CY_{n+1}}^\alpha\le n\}}]\\
&\le \E[\beta_n^2n^{\frac{2}{\alpha}}\Ind_{\{\nrm{\CY_{1}}^\alpha\le n\}}]\end{split}\]
since under $(J2)$, $\{\CY_n:n\ge 1\}$ are i.i.d.
In view of assumption \eqref{L2b},  $\beta_n^2n^{\frac{2}{\alpha}}\le D^2$ and hence we have
\[\begin{split}
\sum_n \E[\beta_n^2\zeta_{n+1}^2]&\le  D^2 \P(\nrm{\CY_{1}}^\alpha\le n)\\
&<\infty\end{split}\]
in view of \eqref{S29}.

Hence $\sum_{n=1}^m \beta_n\zeta_{n+1}$ is an $L^2$-bounded martingale and therefore converges almost surely. Thus
\be\label{S22}
\textstyle\sum_{n=1}^m\beta_n(\nrm{\CV_{n+1}}\Ind_{\{\nrm{\CY_{n+1}}^\alpha \le n\}}-\delta_n) \;\;\text{converges almost surely.}\ee
On the other hand, $\P(\nrm{\CY_{n+1}}^\alpha \ge n)= \P(\nrm{\CY_{1}}^\alpha \ge n)$ (since $\{\CY_n: n\ge 1\}$ are i.i.d.) and so the Borel–Cantelli lemma along with \eqref{L2a} says $\P(\nrm{\CY_{1}}^\alpha \ge n\text{ i.o.})=0$ and hence
\be\label{S23}
\textstyle\sum_{n=1}^m\beta_n(\nrm{\CV_{n+1}}\Ind_{\{\nrm{\CY_{n+1}}^\alpha > n\}}) \;\;\text{converges almost surely.}\ee
Combining \eqref{S22} and \eqref{S23} we again conclude that \eqref{S18} holds, and hence as in the proof above for $(J1)$, it follows that \eqref{V7} holds, completing the proof when $(J2)$ holds. \\

\textbf{Step 4:}\\
Let us now assume that $(J3)$ holds. Let $\alpha\in(1,2]$ be such that \eqref{L3}, \eqref{L4} hold. Note that $\E[\beta_n\nrm{\CY_{n+1}}\Ind_{\{\beta_n\nrm{\CY_{n+1}}> 1\}}]\le \beta_n^\alpha\E[\nrm{\CY_{n+1}}^\alpha]$ $\forall n\ge 1$. As a consequence, 
\[
\begin{split}
\E\left[\textstyle\sum_{n=0}^\infty\beta_n\nrm{\CV_{n+1}}\Ind_{\{\beta_n\nrm{\CY_{n+1}}> 1\}}\right]&\le \E\left[\textstyle\sum_{n=0}^\infty\beta_n\nrm{\CY_{n+1}}\Ind_{\{\beta_n\nrm{\CY_{n+1}}> 1\}}\right]\\
&\le \textstyle\sum_{n=0}^\infty\beta_n^\alpha\E[\nrm{\CY_{n+1}}^\alpha]\\
&\le \textstyle[\sup_n\E[\nrm{\CY_{n+1}}^\alpha]\left[\sum_{k=0}^\infty\beta_k^\alpha\right]\\
&<\infty\end{split}\]
using \eqref{L3}, \eqref{L4}. Hence
\be\label{R11} \textstyle\sum_{n=0}^\infty\beta_n\nrm{\CV_{n+1}}\Ind_{\{\beta_n\nrm{\CY_{n+1}}> 1\}}<\infty \;\;\text{almost surely.}\ee
For $n\ge 1$, let $\delta_n=\E(\nrm{\CV_{n+1}}\Ind_{\{\beta_n\nrm{\CY_{n+1}}\le 1\}}).$ Then
\be\label{R14}\delta_n\le\E(\nrm{\CV_{n+1}}) \le \E(\nrm{\CY_{n+1}})\le \epsilon.\ee
Let us note that (using the definition of $\delta_n$ in the first step, and \eqref{R8} in the second step)
\be\label{R17}
\begin{split}
\textstyle\sum_{n=0}^\infty\beta_n^2\E[(\nrm{\CV_{n+1}}\Ind_{\{\beta_n\nrm{\CY_{n+1}}\le 1\}}-\delta_n)^2]&\le \textstyle\sum_{n=0}^\infty\E[\beta_n^2(\nrm{\CV_{n+1}}\Ind_{\{\beta_n\nrm{\CY_{n+1}}\le 1\}})^2]\\
&\le \textstyle\sum_{n=0}^\infty\E[\beta_n^2\nrm{\CY_{n+1}}^2\Ind_{\{\beta_n\nrm{\CY_{n+1}}\le 1\}}]\\
&\le \textstyle\sum_{n=0}^\infty\E[\beta_n^\alpha\nrm{\CY_{n+1}}^\alpha\Ind_{\{\beta_n\nrm{\CY_{n+1}}\le 1\}}]\\
&\le \textstyle [\sup_{n\ge 0}\E(\nrm{\CY_{n+1}}^\alpha)]\left[\sum_{n=0}^\infty\beta_n^\alpha\right]\\
&<\infty
\end{split}
\ee
where we have used \eqref{L3}, \eqref{L4} and \eqref{U24} at the last step.
Since $\zeta_{n+1}=(\nrm{\CV_{n+1}}\Ind_{\{\beta_n\nrm{\CY_{n+1}}\le 1\}}-\delta_n)$ is a sequence of independent random variables with mean 0, \eqref{R17} implies that $\sum_{k=0}^n\beta_k\zeta_{k+1}$ is an $L^2$-bounded martingale and hence
\[ \textstyle\sum_{n=0}^m\beta_n(\nrm{\CV_{n+1}}\Ind_{\{\beta_n\nrm{\CY_{n+1}}\le 1\}}-\delta_n) \;\;\text{converges almost surely.}\]
Along with \eqref{R11}, this yields
\be\label{R18}
\textstyle\sum_{n=0}^\infty\beta_n(\nrm{\CV_{n+1}}-\delta_n) \;\;\text{converges almost surely.}\ee
Invoking \eqref{R10}, \eqref{R14} and \eqref{R18}, it follows that \eqref{V7} holds, completing the proof when $(J3)$ holds.

\end{document}